\documentclass[12pt]{amsart}
\usepackage[dvips]{color}
\usepackage{graphicx}
\usepackage{epsfig}
\usepackage{tikz}
\usepackage{enumerate}
\usepackage{enumitem}
\usepackage{color}
\usepackage[margin=1.5in]{geometry}
\usepackage{mathrsfs}

\usepackage{latexsym}
\usepackage{amssymb}
\usepackage{amsmath}
\usepackage{amsthm}
\usepackage{amscd}
\theoremstyle{plain}
\newtheorem{thm}{Theorem}[section]
\newtheorem{lem}[thm]{Lemma}
\newtheorem{prop}[thm]{Proposition}

\theoremstyle{definition}
\newtheorem{defin}[thm]{Definition}
\newtheorem{rem}[thm]{Remark}

\newtheorem{conj}[thm]{Conjecture}

\newtheorem{cor}[thm]{Corollary}

\definecolor{purple}{rgb}{.5,0,1}

\newcommand{\R}{\mathbb{R}} 
\newcommand{\Z}{\mathbb{Z}} 
\newcommand{\N}{\mathbb{N}} 

\renewcommand{\limsup}{\varlimsup}

\newcommand{\calA}{{\mathcal{A}}}
\newcommand{\boldtau}{{\boldsymbol{\tau}}}
\newcommand{\boldphi}{{\boldsymbol{\phi}}}

\newcommand{\set}[1]{{\left\{#1\right\}}}
\newcommand{\Sub}{{\mathrm{Sub}}}
\numberwithin{equation}{section}

\newcommand{\TT}{\mathbb{T}}

\begin{document}

\title[Zero Measure Spectrum for Multi-Frequency Operators]{Zero Measure Spectrum for Multi-Frequency Schr\"odinger Operators}

\begin{abstract}
Building on works of Berth\'e--Steiner--Thuswaldner and Fogg--Nous we show that on the two-dimensional torus, Lebesgue almost every translation admits a natural coding such that the associated subshift satisfies the Boshernitzan criterion. As a consequence we show that for these torus translations, every quasi-periodic potential can be approximated uniformly by one for which the associated Schr\"odinger operator has Cantor spectrum of zero Lebesgue measure. We also describe a framework that can allow this to be extended to higher-dimensional tori. 
\end{abstract}

\author[J.\ Chaika]{Jon Chaika}
\address{Department of Mathematics, University of Utah, Salt Lake City, UT 84112, USA}
\email{chaika@math.utah.edu}
\thanks{J.C.\ was supported in part by the Simons foundation, the Warnock chair, and NSF Grant DMS--1452762}

\author[D.\ Damanik]{David Damanik}
\address{Department of Mathematics, Rice University, Houston, TX~77005, USA}
\email{damanik@rice.edu}
\thanks{D.D.\ was supported in part by NSF grant DMS--1700131 and by an Alexander von Humboldt Foundation research award}

\author[J.\ Fillman]{Jake Fillman}
\address{Department of Mathematics, Texas State University, San Marcos, TX 78666, USA}
\email{fillman@txstate.edu}
\thanks{J.F.\ was supported in part by Simons Collaboration Grant \#711663}

\author[P.\ Gohlke]{Philipp Gohlke}
\address{Fakult\"at f\"ur Mathematik, Universit\"at Bielefeld,  Postfach 100131, 33501 Bielefeld, Germany}
\email{pgohlke@math.uni-bielefeld.de}
\thanks{P.G. acknowledges support by the German Research Foundation (DFG) via
the Collaborative Research Centre (CRC 1283)}

\maketitle

\section{Introduction}

This work addresses the persistent occurrence of Cantor spectrum of zero Lebesgue measure in the class of discrete one-dimensional Schr\"odinger operators with generalized quasi-periodic potentials, where the underlying torus has dimension strictly greater than one.

To motivate this problem, let us describe the setting and recall some of the known results. Fix a dimension $d \in \N$ and consider $\alpha \in \TT^d := \mathbb{R}^d/\mathbb{Z}^d$ that is such that the translation $R_\alpha : \TT^d \to \TT^d$, $\omega \mapsto \omega + \alpha$ is minimal. If $g : \TT^d \to \R$ is bounded and measurable, we can consider, for each $\omega \in \TT^d$, the discrete Schr\"odinger operator
$$
[H_{\alpha,g,\omega} \psi](n) = \psi(n+1) + \psi(n-1) + g(\omega + n \alpha) \psi(n)
$$
in $\ell^2(\Z)$. We call such an operator a \emph{generalized quasi-periodic Schr\"odinger operator}. Within this class of sampling functions, one distinguishes several standard regularity classes and observes that the spectral properties of the operators in question depend quite significantly on the chosen regularity class. Standard examples are given by continuous $g$ (this corresponds precisely to the class of \emph{quasi-periodic Schr\"odinger operators}), H\"older continuous $g$, $g$ that are differentiable a certain finite number of times, smooth (i.e., infinitely differentiable) $g$, and analytic $g$.

One is interested in the spectrum and the spectral type. By standard arguments involving the ergodicity of Lebesgue measure with respect to $R_\alpha$, there is a compact set $\Sigma_{\alpha,g}$ such that for Lebesgue almost every $\omega \in \TT^d$, the spectrum of $H_{\alpha,g,\omega}$ is equal to $\Sigma_{\alpha,g}$. Similarly, the spectral type of $H_{\alpha,g,\omega}$ is also Lebesgue-almost surely independent of $\omega$. As we will focus on the spectrum in this paper, we will not go into further details regarding the spectral type and refer the reader to the surveys \cite{D17, MJ17} for background and more information.

The almost sure spectrum $\Sigma_{\alpha,g}$ can have various topological and measure-theoretic properties. It can be a Cantor (i.e., perfect and nowhere dense) set, but it can also be a finite union of non-degenerate compact intervals. The Cantor spectra that occur can have both positive and zero Lebesgue measure. Among those that have zero Lebesgue measure, examples are known with small, and even zero, Hausdorff dimension.

Roughly speaking, when $d = 1$, it is well known how to produce examples with zero Lebesgue measure \cite{DL06a, DL06b} and even zero Hausdorff dimension \cite{LS16}. On the other hand, when $d > 1$, examples are known where the spectrum is a finite union of intervals, and it is (essentially)\footnote{There is a way to recast some known results for primitive substitution subshifts in terms of codings of torus translations; see, for example, \cite{R} for the case of the Tribonacci substitution and \cite{AI} for more examples.} open how to produce spectra of zero Lebesgue measure. The present paper develops a way of producing many such examples. Indeed they are ``ample'' in a way we will make precise.

Since we used zero Lebesgue measure and non-Cantor structure to distinguish between the two cases $d = 1$ and $d > 1$ in the previous paragraph, let us point out that proving the genericity of Cantor spectrum in $C(\TT^d)$ for any fixed minimal translation $R_\alpha$ (without supplying any information about the Lebesgue measure of the set) has a proof that works simultaneously for all values of $d \in \N$; see \cite{ABD09, ABD12}. On the other hand, in the analytic category, Cantor spectrum is typical when $d = 1$ (the literature is extensive; see, e.g., \cite{E91, GS11, P04}, and the surveys \cite{D17, MJ17} for a more complete list), while it is not typical when $d > 1$ (at least in the large coupling regime \cite{GSV19}).

To summarize, the mechanisms leading to Cantor spectrum of zero Lebesgue measure in the context of generalized quasi-periodic Schr\"odinger operators are quite well understood in the one-frequency case ($d = 1$), but so far they are poorly understood in the multi-frequency case ($d > 1$). We will discuss a mechanism here that works in the multi-frequency case, which leads to a class of examples that is in some ways as rich and ``ample'' as the existing work in the one-frequency case.

\begin{defin}
A function $g : \TT^d \to \R$ is called \emph{elementary} if it is measurable and takes finitely many values. The set of elementary functions $g : \TT^d \to \R$ is denoted by $\mathcal{E}(\TT^d)$. A subset of $\mathcal{E}(\TT^d)$ is called \emph{ample} if its $\|\cdot\|_\infty$-closure in $L^\infty(\TT^d)$ contains $C(\TT^d)$.
\end{defin}

\begin{thm}\label{t.mfqpsoapplication}
Let $d = 2$. Then, for Lebesgue almost every $\alpha \in \TT^d$, the set
$$
\mathcal{Z}_\alpha = \{ g \in \mathcal{E}(\TT^d) :  \Sigma_{\alpha,g} \text{ is a Cantor set of zero Lebesgue measure} \}
$$
is ample.
\end{thm}

\begin{rem}
(a) In the case $d = 1$, this is a result of Damanik-Lenz \cite{DL06a, DL06b}. Specifically, it follows by combining \cite[Theorem~2]{DL06a} and \cite[Theorem~10]{DL06b}. Actually, in this case, the full measure set of $\alpha \in \TT$ is explicit: it is the set of all irrational numbers. By contrast, the full measure set in Theorem~\ref{t.mfqpsoapplication} is not explicit.

(b) The fact that the result can be extended to a value of $d$ that is greater than one is not obvious, and indeed surprising, since the straightforward extension of \cite[Theorem~10]{DL06b} is known to fail, compare Remark~\ref{r.super-linear} below.

(c) The proof of Theorem~\ref{t.mfqpsoapplication} also employs \cite[Theorem~2]{DL06a}, but replaces the use of \cite[Theorem~10]{DL06b} by a more sophisticated process to verify the assumption of \cite[Theorem~2]{DL06a}.

(d) To the best of our knowledge, there is no known example of a quasi-periodic multi-frequency potential (i.e., $d > 1$ and $g \in C(\TT^d)$) so that the associated Schr\"odinger operator has zero-measure spectrum. It is unclear whether such an example exists. The fact that arbitrarily small $\|\cdot\|_\infty$ perturbations of an arbitrary $g \in C(\TT^d)$ can produce this effect is therefore interesting.

(e) We described the occurrence of zero-measure spectrum obtained via this route as ``persistent'' above, so let us explain what we mean by that. The $g \in \mathcal{E}(\TT^2)$ we obtain for which $\Sigma_{\alpha,g}$ is a Cantor set of zero Lebesgue measure are actually such that $\Sigma_{\alpha,\lambda g}$ is a Cantor set of zero Lebesgue measure for every $\lambda \in \R$ with $\lambda \not= 0$. Thus the phenomenon is persistent with respect to varying the coupling constant. This should be contrasted with the fact that any known $g \in C(\TT)$ for which $\Sigma_{\alpha,g}$ has been shown to have zero Lebesgue measure for suitable (irrational) $\alpha \in \TT$ is such that $\Sigma_{\alpha,\lambda g}$ has positive Lebesgue measure for every $\lambda \in \R$ with $|\lambda| \not= 1$. In other words, the zero-measure property is highly unstable with respect to a variation of the coupling constant in the quasi-periodic setting.

(f) We regard it as an interesting open problem to explore whether Theorem~\ref{t.mfqpsoapplication} can be extended to some larger values of $d$. Several components of our proof of Theorem~\ref{t.mfqpsoapplication} indeed do extend to values of $d$ greater than $2$. In the final section of this paper we comment on why our result is limited to the case $d = 2$ and point out the obstacles one needs to overcome if one wants to prove a result for some $d > 2$.
\end{rem}

The remainder of the paper is organized in the following way. We collect some necessary background in Section~\ref{sec:prelim}, including known results about multidimensional continued fraction algorithms and S-adic subshifts. In Section~\ref{sec:bosh}, we prove a sufficient criterion for an S-adic subshift to obey Boshernitzan's criterion for unique ergodicity.  Building on \cite{BST20}, we apply this criterion in Section~\ref{sec:2D} to deduce that Boshernitzan's criterion holds for certain subshifts arising from suitable two-dimensional continued fraction algorithms. We conclude the proof of Theorem~\ref{t.mfqpsoapplication} in Section~\ref{sec:mainproof}. Finally, we discuss the case $d \geq 3$ in Section~\ref{sec:higher}, including the overall strategy that one should implement as well as the obstacles that one must overcome in order to apply said strategy.

\section{Preliminaries} \label{sec:prelim}

\subsection{Multi-Dimensional Continued Fraction Algorithms}

\subsubsection{Motivation and Notation}

Continued fractions are a tool to understand the Diophantine properties of numbers and the dynamical properties of rotations. The theory has been best developed in dimension one where the Euclidean algorithm and its acceleration, the Gauss map, are incredibly useful. There are many generalizations of these algorithms to higher dimensions. For our purposes we will restrict our attention to the Cassaigne-Selmer algorithm and the Brun algorithm (the latter in the special case of four dimensions).

\subsubsection{The Cassaigne-Selmer Algorithm}

Denote $\mathbb{R}_+ =[0,\infty)$ and let
$$
\Delta = \Delta_3=\{(x_1,x_2,x_3)\in \mathbb{R}_+^3: x_1+x_2+x_3=1\}.
$$
The \emph{Cassaigne-Selmer algorithm} is given by
\begin{equation} \label{eq:CassSelDef}
T_C : \Delta \to \Delta \text{ by }T(x_1,x_2,x_3) = \begin{cases}(\frac{x_1-x_3}{x_1+x_2},\frac{x_3}{x_1+x_2},\frac{x_2}{x_1+x_2}) & \text{ if }x_1\geq x_3\\[2mm]
(\frac{x_2}{x_2+x_3},\frac{x_1}{x_2+x_3},\frac{x_3-x_1}{x_2+x_3}) & \text{ if }x_3>x_1.\end{cases}
\end{equation}
This algorithm was studied in \cite{CLL} for its connection to word combinatorics. There is an ergodic $T_C$-invariant probability measure $\nu_C$ on $\Delta$ which is equivalent to Lebesgue measure. Indeed, the Cassaigne-Selmer algorithm is conjugate to the Selmer algorithm \cite{CLL}. This algorithm is ergodic by \cite[Section 7]{Schw00}, whose argument presenting the proof of ergodicity of the fully sorted Selmer algorithm generalizes to show that the semi-sorted Selmer algorithm is ergodic.

\subsubsection{The Brun Algorithm for $d=4$}

Let
$$
\Delta = \Delta_4 = \{(x_1,x_2,x_3,x_4)\in \mathbb{R}_+^4: x_1+x_2+x_3 + x_4=1\}
$$
and, for $i,j \in \{1,2,3,4\}$, let
$$
\Delta(i,j) = \{ (x_1, x_2,x_3, x_4) : x_i \geq x_j \geq x_k \mbox{ for all } k \notin \{i,j\} \}.
$$
The \textit{Brun algorithm} $T_B \colon \Delta \to \Delta$ is defined for $(x_1, \ldots, x_4) \in \Delta(i,j)$ as
\[
T_B(x_1,\ldots, x_4)_k
= \begin{cases}
\frac{x_k}{1- x_j} & \text{ if } k \neq i, \\[2mm]
\frac{x_i - x_j}{1 - x_j} & \text{ if } k = i .
\end{cases}
\]
This map is well-defined almost everywhere on $\Delta$.
The ergodicity of this algorithm follows as in \cite{Schw00}. Hence, there exists an ergodic $T_B$-invariant probability measure $\nu_B$ on $\Delta$, which is equivalent to Lebesgue measure.

\subsection{S-Adic Subshifts} \label{ssec:sadic}

Given a finite set $\calA$, give the full shift $\calA^\Z$ the product topology inherited from placing the discrete topology on each factor, and define the \emph{shift map} $S:\calA^\Z \to \calA^\Z$ by $[Sx](n) = x(n+1)$.
A \emph{subshift} over $\calA$ is a closed (hence compact) $S$-invariant subset $X\subseteq \calA^\Z$.

The free monoid will be denoted $\calA^* = \bigcup_{n= 0}^\infty \calA^n$; the unique element of $\calA^0$ is denoted $\varepsilon$ and called the empty word; the length of $u \in \calA^n$ is $|u| = n$.
Write
\begin{equation}\label{eq:subworddef}
\#_u(v) := \#\{ j : v_{j+1} v_{j+2} \cdots v_{j+|u|} = u\}
\end{equation}
for the number of times $u$ occurs in $v$, $u\triangleleft v$ if $\#_u(v)>0$, and  $L(u)$ for the set of all subwords of $u \in \calA^*$, $\calA^\N$ or $\calA^\Z$. For a subshift $X$, the \emph{language} of $X$ is
$$
L(X) : =\{u : u \in L(x) \text{ for some } x \in X\}.
$$
When $(X,S)$ is minimal, $L(X) = L(x)$ for every $x \in X$.

\begin{defin}\label{def:Bosh}
Let $(X,S)$ be a minimal subshift. We say that $(X,S)$ satisfies the \emph{Boshernitzan criterion} if there exist an $S$-invariant probability measure $\mu$, a constant $C>0$, and a sequence $n_1, n_2, \ldots \to \infty$  so that for all $w=w_1 \cdots w_{n_i} \in L(X)$,
$$
\mu( \{ x \in X : x_1 \cdots x_{n_i} = w \} > \frac C {n_i}.
$$
\end{defin}

A \emph{substitution} is an endomorphism $\tau: \calA^* \to \calA^*$, which is uniquely defined by its values on individual letters of $\calA$. We shall also assume that all substitutions are \emph{non-erasing} in the sense that $\tau(a)\neq \varepsilon$ for every $a \in \calA$, and denote the set of non-erasing substitutions on $\calA$ by $\Sub(\calA)$. For each $\tau \in \Sub(\calA)$, one associates the \emph{substitution matrix} $M = M_{\tau} \in \mathrm{End}(\Z^\calA)$, with entries given by
\[ M_\tau[a,b] = \#_a (\tau(b)).\]

 An \emph{S-adic system} over $\calA$ is defined by a choice of a \emph{directive sequence} $\boldtau = (\tau_n)_{n=0}^\infty$  of substitutions on $\calA$.  We will encounter products quite frequently, so, for $0\le m < n$, we write
 \[\tau_{[m,n]} = \tau_m \cdots \tau_n,\]
 with obvious conventions for open and half-open intervals.
 For $a \in\calA$, write $w_n(a) = \tau_{[0,n]}(a)$.
  Similarly, for the substitution matrices, we write $M_I = M_{\tau_I}$ for an interval $I$. Clearly, for $I = [m,n]$, one has
 \[M_{[m,n]} =  M_{\tau_m} M_{\tau_{m+1}} \cdots M_{\tau_{n}}.\]
  The \emph{language} associated to $\boldtau$ is
 \[L(\boldtau) := \set{ w \in \calA^* : w \triangleleft w_n(a) \text{ for some } a \in \calA \text{ and } n \in \N_0}. \]
 We also call this the set of allowed words.
It is easy to check that
 \[ X = X(\boldtau) := \set{ x \in \calA^\Z : L(x) \subseteq L(\boldtau)}, \]
 is a non-empty subshift, provided that
 \[\lim_{n \to \infty} \max_{a \in \calA} |w_n(a)| = \infty.\]
In this case, we call $X(\boldtau)$ the S-adic subshift generated by $\boldtau$.

\subsection{S-Adic Subshifts Related to Multi-Dimensional Continued Fractions}
Both the Cassaigne-Selmer algorithm and the Brun algorithm are of the form
\[
T \colon \Delta \to \Delta, \quad \bold{x} \mapsto \frac{A(\bold{x})^{-1} \bold{x}}{\|A(\bold{x})^{-1} \bold{x}\|_1}
\]
for some locally constant matrix valued function $A \colon \Delta \to \operatorname{GL}(d,\Z)$. Following \cite{BST20}, we select for each $\bold{x} \in \Delta$ a substitution $\varphi(\bold{x})$ on the alphabet $\mathcal A = \{1,\ldots, d\}$ such that $A(\bold{x})$ coincides with the substitution matrix $M_{\varphi(\bold{x})}$. In the case of the Cassaigne-Selmer algorithm this is achieved by
\[
\varphi(\bold{x}) =
\begin{cases}
\gamma_1 \quad \text{ if } x_1 \geq x_3,
\\ \gamma_2 \quad \text{ if } x_3 > x_1,
\end{cases}
\]
with the Cassaigne-Selmer substitutions
\[
\gamma_1 \colon \begin{cases}
1 \mapsto 1
\\ 2 \mapsto 13
\\ 3 \mapsto 2
\end{cases}
\quad \quad
\gamma_2 \colon \begin{cases}
1 \mapsto 2
\\ 2 \mapsto 13
\\ 3 \mapsto 3.
\end{cases}
\]
For the Brun algorithm we consider the class of substitutions
\[
\beta_{ij} \colon j \mapsto ij, \; k \mapsto k \mbox{ for } k \in \mathcal{A} \setminus \{j\}.
\]
for  $i,j \in \mathcal A = \{1,2,3,4\}$ and we set $\varphi(\bold{x}) = \beta_{ij}$ for $ \bold{x} \in \Delta(i,j)$.

Given a substitution selection $\varphi:\Delta \to \Sub(\calA)$, the orbit of a point $\bold{x} \in \Delta$ under the action of $T$ defines an S-adic system, called a \emph{substitutive realization} of $(\Delta,T,A)$, given by the directive sequence
\[
\boldphi(\bold{x}) = (\varphi(T^n \bold{x}))_{n=0}^\infty.
\]
The corresponding subshift is given by $(X(\boldphi(\bold{x})),S)$.
On the other hand, we relate to each point $\bold{x}$ in the $d$-dimensional simplex $\Delta$ a point on the torus $\TT^{d-1}$ by the map $\pi: \Delta \to \TT^{d-1}$, which denotes the projection to the first $d-1$ coordinates. Note that $\pi$ is not a surjective map but for
$$
\TT^{d-1}_{\Delta} = \{ t \in \TT^{d-1} : t_1 + \ldots + t_{d-1} \leq 1\},
$$
the map $\pi \colon \Delta \to \TT^{d-1}_{\Delta}$, $\bold{x} \mapsto \pi(\bold{x})$ is a bijection, identifying $\TT^{d-1} \cong [0,1)^{d-1}$ in the obvious fashion. Slightly abusing notation, we use the same symbol, $\pi$, to denote both maps.

\subsection{Natural Codings of Torus Translations}

For the $d$-dimensional torus $\TT^d$ and $\alpha \in \TT^d$, let $R_\alpha \colon \TT^d \to \TT^d$, $R_\alpha(\omega) = \omega + \alpha$ denote the torus translation associated to $\alpha$.

We present in the following a weaker version of the term \emph{natural coding} as defined in \cite{BST20}. This turns some of the results we cite from \cite{BST20} into mere corollaries which are, however, sufficient for our purposes. A collection $\mathscr{F} = \{\mathcal F_1, \ldots, \mathcal F_h\}$ is called a \emph{natural measurable} partition of $\TT^d$ if $\bigcup_{i=1}^h \mathcal F_i = \TT^d$, $\mathcal{F}_j \cap \mathcal{F}_k$ has zero measure for each $j \neq k$, and each $\mathcal F_i$ is measurable with dense interior and zero measure boundary. Given the map $R_\alpha$, the language associated with $\mathscr{F}$, denoted $L(\mathscr{F})$, is the set of finite words $w = w_0 \cdots w_n \in \{1,\ldots, h\}^{\ast}$ such that $\bigcap_{k=0}^{n} R_\alpha^{-k} \mathring{\mathcal F}_{w_k} \neq \emptyset$, where $\mathring{A}$ denotes the interior of $A$.
\begin{defin}
A subshift $(X,S)$ is called a \emph{natural coding} of $(\TT^d,R_\alpha)$ if its language coincides with the language of a natural measurable partition $\{\mathcal F_1,\ldots,\mathcal F_h\}$ and
$$
\bigcap_{n \in \N} \overline{\bigcap_{k=0}^{n} R_\alpha^{-k} \mathring{\mathcal F}_{x_k}}
$$
consists of a single point for every $x=(x_n)_{n \in \Z} \in X$.
\end{defin}

The following result concerning the Cassaigne-Selmer algorithm is essential for our analysis.

\begin{prop}\label{PROP:nat coding} \cite[Theorem 6.2]{BST20}
Let $\boldphi$ be the substitutive realization of the Cassaigne-Selmer algorithm. For $\nu_C$-almost every $\bold{x} \in \Delta$, the subshift $(X({\boldphi(\bold{x}))},S)$ is a natural coding of $(\TT^2, R_{\pi(\bold{x})})$.
\end{prop}
Note that \cite[Theorems A and B]{FN20} are closely related results, that would have also been sufficient for our purposes.

\begin{rem}
\label{REM:T-Delta-to-T}
If $\mathscr{F} = \{\mathcal F_1, \ldots, \mathcal F_h\}$ is a natural measurable partition of $\TT^2$ and $M \in \mathrm{GL}(2,\Z)$, then the language generated by $R_\alpha$ on $\mathscr{F}$ coincides with the language generated by $R_{M\alpha}$ on the natural measurable partition $M\mathscr{F} =  \{M\mathcal F_1, \ldots, M \mathcal F_h\}$. In particular, if $(X,S)$ is a natural coding of $(\TT^2, R_\alpha)$, then it is also a natural coding of $(\TT^2, R_{M\alpha})$. In two dimensions we could simply take $M\alpha := -\alpha$, to obtain natural codings for (almost) all $\alpha \in \TT^2$ from codings for $\alpha \in \TT^2_{\Delta}$. For the more general $d$-dimensional cases, we still obtain $\TT^d$ from $\TT^d_{\Delta}$ via general linear transformations, compare \cite[Rem.~3.5]{BST20}.
\end{rem}
One would naturally like to obtain analogs of Proposition~\ref{PROP:nat coding} for higher-dimensional torus translations. For such translations, the Brun algorithm is a natural candidate to use for the associated continued fraction algorithm. However, in that case, there is a technical ingredient (namely negativity of the second Lyapunov exponent) which is currently unclear. We discuss this in more detail in Section~\ref{sec:higher}.

\subsection{Zero-Measure Spectrum via the Boshernitzan Criterion}

Given a finite alphabet $\calA$ and a subshift $X \subseteq \calA^\Z$, one can define Schr\"odinger operators in $\ell^2(\Z)$ by generating potentials which are obtained through real-valued sampling along the $S$-orbits of $X$. That is, if $f : X \to \R$ is given, we associate with each $x \in X$ the potential $V_x : \Z \to \R$ given by $V_x(n) = f(S^n x)$, $n \in \Z$. The Schr\"odinger operator $H_x$ in $\ell^2(\Z)$ is then given by
$$
[H_x \psi](n) = \psi(n+1) + \psi(n-1) + V_x(n) \psi(n).
$$
One typically restricts attention to locally constant functions $f$, that is, functions that depend on only finitely many entries of the input sequence $x$. Such functions are of course continuous, but in addition they preserve the finite-valuedness, which is crucial to many arguments in the study of these operators.

If $X$ is minimal and $f$ is locally constant, then a simple strong approximation argument shows that there is a compact set $\Sigma_{X,f} \subset \R$ such that $\sigma(H_x) = \Sigma_{X,f}$ for every $x \in X$. Obviously, a minimal subshift $X$ is finite if and only if every $V_x$ is periodic, and in this case $\Sigma_{X,f}$ is well known to be a union of finitely many non-degenerate compact intervals. Similarly, if $f$ is constant, the same conclusions hold. Ruling out these degenerate cases, it is an interesting question whether $\Sigma_{X,f}$ must have zero Lebesgue measure. In fact, Simon conjectured that this must be the case in complete generality, but this conjecture has been disproved in \cite{ADZ14}.

On the other hand, the Boshernitzan criterion turns out to be a sufficient condition \cite[Theorem~2]{DL06a}:
\begin{thm}\label{t.BimpliesZM}
If the minimal subshift $X$ satisfies the Boshernitzan criterion and $f$ is locally constant, then either all $V_x$ are periodic or the set $\Sigma_{X,f}$ is a Cantor set of zero Lebesgue measure.
\end{thm}

\section{S-Adic Subshifts Satisfying the Boshernitzan Criterion} \label{sec:bosh}

Let $\boldtau = (\tau_k)_{k=0}^\infty$ be a directive sequence generating an S-adic system, $(X(\boldtau),S)$. Refer to Section~\ref{ssec:sadic} for definitions and notation. Our key auxiliary result is a sufficient criterion on $\boldtau$ for $(X(\boldtau),S)$ to satisfy Boshernitzan's criterion for unique ergodicity.
\begin{defin}
For $a,b \in \calA$, we say that $a$ \emph{precedes} $b$ \emph{at level $n$} if there are $m \in \N$ and $c \in \calA$ such that $ab \triangleleft \tau_{[n+1,n+m]}(c)$.  For an interval $I = [n+1,n+\ell]$, we say $\tau_I$ is a \emph{word builder} at level $n$ if, whenever $a$ precedes $b$ at level $n$, there is $c \in \calA$ such that $ab \triangleleft \tau_I(c)$.
\end{defin}

\begin{thm} \label{thm:S-adic Boshernitzan}
Suppose there exists a constant $N>0$ so that, for infintely many $n_0$, there exist $n_0 < n_1 < n_2 < n_3$ so that
\begin{enumerate}[label = {\rm(\alph*)}]
\item $M_{[n_0+1,n_1]}$ and $M_{[n_2+1,n_3]}$ are positive matrices,
\smallskip

\item $\tau_{[n_1+1,n_2]}$ is a word builder at level $n_1$,
\smallskip

\item $\max\{ \|M_{[n_0+1,n_1]}\|, \|M_{[n_1+1,n_2]}\|, \|M_{[n_2+1,n_3]}\|\} \leq N$.
\smallskip

\end{enumerate}
Then $(X(\boldtau),S)$ satisfies Boshernitzan's criterion.
\end{thm}

\begin{lem}\label{lem:concatenation}
If $\tau_{n+1}(a)=b_1 \ldots b_r$, then $w_{n+1}(a)=w_n(b_1) \ldots w_n(b_r)$.
\end{lem}

\begin{proof} This follows immediately from the definition of $w_n$.\end{proof}

\begin{cor}\label{cor:length bound}
Let $n,k \in \N$. If $M_{[n,n+k]}$ is a positive matrix, then, for all $a, a' \in \calA$, one has
$$
\frac{|w_{n+k}(a)|}{|w_{n+k}(a')|} \leq \underset{i,j,j'}{\max}\,\left\{\frac{M_{[n,n+k]}[i,j]}{M_{[n,n+k]}[i,j']} \right\}.
$$
\end{cor}

\begin{proof}
For each $b \in \calA$, we apply Lemma \ref{lem:concatenation} $k$ times to write $w_{n+k}(b)$ as a concatenation of $w_n(a)$ for $a\in \calA$. For each $i,j,j' \in \calA$, the ratio of occurrences of $w_n(i)$ in such a decomposition of $w_{n+k}(j)$ and $w_{n+k}(j')$ is at most the right hand side.
\end{proof}

\begin{lem}\label{lem:all words}
If $\tau_{[n+1,n+\ell]}$ is a word builder at level $n$, then every allowed word of length at most $\underset{c \in \mathcal{A}}{\min}\,  |w_n(c)|$ is a subword of $w_{n+\ell}(c)$ for some $c \in \calA$.
\end{lem}

\begin{proof}
Every word is a truncation of concatenations of $w_n(c)$ as $c$ varies in $\mathcal{A}$. So every word of length at most $\underset{c \in \mathcal{A}}{\min}\,  |w_n(c)|$ is formed by concatenating a (possibly empty) suffix of $w_n(a)$ with a (possibly empty) prefix of $w_n(a')$ where $a$ precedes $a'$ at level $n$.
All such combinations appear in $w_{n+\ell}(c)$ for some $c \in \mathcal{A}$.
\end{proof}

\begin{lem}\label{lem:meas bound}
If $\tau_{[n+1,n+\ell]}$ is a word builder and $M_{[n+\ell+1,n+\ell+k]}$ is positive, then the measure of the cylinder set associated with any word of length $\min_{a\in \mathcal{A}}|w_n(a)|$ is at least
$$
\left(\underset{c \in \mathcal{A}}{\max}\, |w_{n+\ell+k}(c)|\right)^{-1}.
$$
\end{lem}

\begin{proof}
Every allowed word of length at most $\underset{a \in \mathcal{A}}{\min}\, |w_n(a)|$ appears at least once in every $w_{n+\ell+k}(c)$. Indeed, every $c$ appears in $\tau_{n+\ell+1} \cdots \tau_{n+\ell+k}(a)$ by the positivity of $M_{n+\ell+1} \cdots M_{n+\ell+k}$. So every $w_{n+\ell}(c)$ appears in every $w_{n+\ell+k}(a)$. By Lemma~\ref{lem:all words} this implies that every allowed word of length at most $\underset{a \in \mathcal{A}}{\min}\, |w_n(a)|$ appears at least once in every $w_{n+\ell+k}(c)$.

So the proportion of every allowed word in such blocks is at least $(\underset{c \in \mathcal{A}}{\max}\, |w_{n+\ell+k}(c)|)^{-1}$. As our language is a concatenation of $w_{n+\ell+k}(c)$ as $c$ varies in $\mathcal{A}$ we have the claim.
\end{proof}

\begin{proof}[Proof of Theorem~\ref{thm:S-adic Boshernitzan}]
This follows from Lemma \ref{lem:meas bound} and Corollary \ref{cor:length bound}. Indeed,
$$
\underset{c \in \mathcal{A}}{\max}\, |w_{n_3}(c)|
\leq N^2 \left(\underset{c \in \mathcal{A}}{\max}\, |w_{n_1}(c)|\right)
\leq N^3\left(\underset{c \in \mathcal{A}}{\min}\, |w_{n_1}(c)|\right).
$$
So we have that the measure of any cylinder of length
$ \underset{c \in \mathcal{A}}{\min}\, |w_{n_1}(c)|$ is at least $(N^3 \underset{c \in \mathcal{A}}{\min}\, |w_{n_1}(c)|)^{-1}$. Consequently, there exist infinitely many $r$ so that  we satisfy the Boshernitzan criterion with $C=(N^3r)^{-1}$. 
\end{proof}

\begin{rem}\label{r.super-linear}
It is easy to see that any subshift satisfying the Boshernitzan criterion must have a complexity function that is linearly bounded on a subsequence. This in turn shows that for codings of higher-dimensional torus translations, care must be taken if there is to be any hope to generate subshifts satisfying the Boshernitzan criterion. Indeed, it is known that any coding of a minimal translation of $\TT^d$, $d \ge 2$, relative to a partition of $\TT^d$ into sufficiently nice sets has a super-linear lower bound; compare, for example, \cite{C,SW}.
\end{rem}

\section{2D Toral Translations} \label{sec:2D}

The substitution matrices associated to the Cassaigne--Selmer substitutions $\gamma_1$ and $\gamma_2$ are given by
\[
C_1 = \begin{bmatrix}
1 & 1 & 0\\
0 & 0 & 1\\
0 & 1 & 0
\end{bmatrix}
\quad \mbox{and} \quad
C_ 2 = \begin{bmatrix}
0 & 1 & 0 \\
1 & 0 & 0 \\
0 & 1 & 1
\end{bmatrix},
\]
respectively. Recall that $\nu_C$ denotes the $T_C$-ergodic measure on $\Delta$ which is equivalent to Lebesgue measure. For the remainder of this section, let $T = T_C$ and $\nu = \nu_C$.
The pushforward of Lebesgue measure on $\Delta$ under $\pi$ is equivalent to Lebesuge measure (and therefore to $\nu$) on $\TT^2_{\Delta}$. Hence, for almost all $\alpha \in \TT^2_{\Delta}$, the subshift $(X(\boldphi(\pi^{-1}(\alpha))),S)$ is a natural coding of $(\TT^2, R_\alpha)$ due to Proposition~\ref{PROP:nat coding}.

\begin{prop}
\label{PROP:T^2-codings}
For Lebesgue a.e.\ $\alpha \in \TT^2_{\Delta}$, the subshift $(X(\boldphi(\pi^{-1}(\alpha))),S)$ satisfies Boshernitzan's criterion. In particular, for almost every $\alpha \in \TT^2$, the toral translation $(\TT^2, R_\alpha)$ admits a natural coding that satisfies Boshernitzan's criterion.
\end{prop}

\begin{proof}
It suffices to show that for $\nu$-almost every $\bold{x} \in \Delta$, the subshift $(X(\boldphi(\bold{x})),S)$ satisfies Boshernitzan's criterion. Note that $\tau = \gamma_1 \circ \gamma_2$ is a primitive substitution, indeed $M_{\tau}^3$ is positive.

Further we claim that the substitution $\tau' = \gamma_1^2 \gamma_2 \gamma_1 \gamma_2^3 \gamma_1$ is a word builder, irrespective of its position within a directive sequence $(\tau_n)_{n=0}^{\infty} \in \{ \gamma_1, \gamma_2 \}^{\N_0}$. To verify this, we first observe that the set $\gamma^2_1 (\mathcal A^2)$ does not contain any of the words in $\{22,23,32,33\}$ as a subword. Hence, whenever $\tau' = \tau_{[n+1,n+8]}$ and $a$ precedes $b$ at level $n$, it follows that $ab \in \mathcal L_2 := \{11,12,13,21,31\}$. A direct calculation yields that $\tau'(1) = 1213113$ and so for all $ab \in \mathcal L_2$ we find that $ab \triangleleft \tau'(1)$. In particular, $\tau'$ is a word builder.
The substitution $\tau^{\ast} = \tau^3 \tau' \tau^3$ is a composition of $\ell = 14$ substitutions drawn from $\{\gamma_1, \gamma_2\}$.
Let
\[
B_m = \{ (\tau_n)_{n \in \N_0} \in \{ \gamma_1, \gamma_2\}^{\N_0} : \tau_{m} \circ \ldots \circ \tau_{m + \ell -1} = \tau^{\ast} \}
\]
and $B = \limsup_{m \to \infty} B_m$. By Theorem~\ref{thm:S-adic Boshernitzan}, for every $\boldtau \in B$, the corresponding subshift $(X(\boldtau),S)$ satisfies Boshernitzan's criterion. Hence, it is enough to show that $\mu = \nu \circ \boldphi^{-1}$ assigns full measure to $B$. We consider the set
\begin{align*}
D & = \boldphi^{-1} (B_0) = \{ \bold{x} \in \Delta : A(\bold{x}) \cdots A(T^{ \ell - 1}\bold{x}) = M_{\tau^{\ast}} \}
\end{align*}
Since the map $\boldphi$ conjugates $T$ and $S$, we have that
\[
\boldphi^{-1}(B_m) = \boldphi^{-1}(S^{-m}B_0) = T^{-m} \boldphi^{-1}(B_0) = T^{-m} D
\]
for all $m \in \N_0$. By Birkhoff's ergodic theorem, we have for almost every $\bold{x} \in \Delta$ that
\[
\lim_{n \to \infty} \frac{1}{n} \sum_{m = 0}^{n-1} \mathbf{1}_{\boldphi^{-1}(B_m)}(\bold{x}) =
\lim_{n \to \infty} \frac{1}{n} \sum_{m = 0}^{n-1} \mathbf{1}_{D}(T^m \bold{x}) = \nu(D).
\]
If $\nu(D) >0$, we therefore conclude that almost-every $\bold{x}$ is contained in infinitely many $\boldphi^{-1}(B_m)$ and hence in $\boldphi^{-1}(B)$, implying $\nu(\boldphi^{-1}(B)) =1$.
It remains to show that $\nu(D) > 0$. 

 Let $\Delta(1) = \{ x \in \Delta : x_1 \geq x_3 \}$ and $\Delta(2) = \{ x \in \Delta : x_3 > x_1 \}$, that is, \[A(\bold{x}) = C_i \iff \bold{x} \in \Delta(i).\]
In the following, we identify sets that coincide up to a set of Lebesgue measure zero---this applies in particular to the boundaries of the sets $\Delta$, $\Delta(1)$ and $\Delta(2)$.
Since $T(\Delta(i)) = \Delta$ and $T$ acts on $\Delta(i)$ as the radial projection of $C_i^{-1}(\Delta(i))$ to $\Delta$, we obtain that the radial projection of $C_i(\Delta)$ to $\Delta$ coincides with $\Delta(i)$.
Abusing notation slightly, we use $C_i$ to also denote the projective action of $C_i$ on $\Delta$.
With this convention, it is straightforward to check that $A(\bold{x}) = C_i$ if and only if $\bold{x} \in C_i(\Delta)$
(note that here we could also replace $\Delta$ with the positive cone).
Similarly, one has $A(\bold{x}) A(T \bold{x}) = C_i C_j$ precisely if $\bold{x} \in C_i(\Delta)$ and $T \bold{x} \in C_j(\Delta)$, where $T = C_i^{-1}$ in this case. That is, we have equivalence to $\bold{x} \in C_i(\Delta)$ and $\bold{x} \in C_i C_j(\Delta)  \subset C_i(\Delta)$. Inductively, we find that $A(\bold{x}) \cdots A(T^{k} \bold{x}) = C_{i_0} \cdots C_{i_k}$ if and only if
$
\bold{x} \in C_{i_0} \cdots C_{ i_k} (\Delta).
$
For our case at hand, we obtain that
$x \in D$ if and only if $\bold{x} \in M_{\tau^{\ast}} (\Delta)$. Note that, as  $M_{\tau^{\ast}}$ is primitive, it acts as a projective contraction on the positive cone.
Since each of $C_1, C_2$ is invertible, so is $M_{\tau^{\ast}}$ and the set $M_{\tau^{\ast}} (\Delta)$ has positive Lebesgue measure. It follows that the Lebesgue measure (and hence the $\nu$-measure) of $D$ is positive.
Finally, to go from $\alpha \in \TT^2_{\Delta}$ to more general $\alpha \in \TT^2$, we make use of Remark~\ref{REM:T-Delta-to-T}.
\end{proof}

\section{Proof of Theorem~\ref{t.mfqpsoapplication}} \label{sec:mainproof}

In this section we derive Theorem~\ref{t.mfqpsoapplication} from our work in the previous sections. Let us begin with a discussion of elementary functions on $\TT^d$ and how they relate to locally constant functions on $(X,S)$, where $(X,S)$ is a natural coding of $R_\alpha$ associated with the natural measurable partition $\{\mathcal{F}_1,\ldots,\mathcal{F}_h\}$.  We define $\eta:X \to \TT^d$ by $\eta(x) = \omega$, where $\omega$ is the unique point in
\[\bigcap_{n \in \N} \overline{\bigcap_{k=0}^{n} R_\alpha^{-k} \mathring{\mathcal F}_{x_k}}.\]
Let \[\mathcal{G} = \bigcap_{k \in \Z} R_\alpha^{-k}\left[\bigcup_{j=1}^h \mathring{\mathcal{F}}_j\right],\]
which is a dense $G_\delta$ set of full Lebesgue measure in $\TT^d$ (by definition of natural coding).
For $\omega \in \mathcal{G}$, we can invert this by mapping $\omega$ to $x = (x_k)_{k\in\Z}$ given by $R_\alpha^k \omega \in \mathring{\mathcal{F}_{x_k}}$.

Given $w = w_0 \cdots w_n \in L(X)$, let
\[ \mathcal{F}_w =\bigcap_{k=0}^n R_\alpha^{-k} \mathcal{F}_{w_k},\]
which is nonempty by the definition of $L(X)$.
Let $\chi_{w}$ denote the characteristic function of $\mathcal{F}_w$, and let $\mathscr{A}$ denote the algebra generated by $\{\chi_{w}: w \in L(X)\}$.

\begin{prop} \label{prop:locconstAmple}
If $(X,S)$ is a natural coding of $R_\alpha$, then $\mathscr{A}$ is ample. In particular, $\mathscr{A}\setminus\{\mathrm{constants}\}$ is ample as well.
\end{prop}

\begin{proof}
Given $f \in C(\TT^d)$ and $\varepsilon>0$, find $\delta > 0$ so that $|f(\theta_1)-f(\theta_2)|<\varepsilon$ whenever $\mathrm{dist}(\theta_1,\theta_2)<\delta$. Choose $n$ large enough that for any $w \in L(X)$ of length $n$, $\mathrm{diam}(\mathcal{F}_w)< \delta$, and define
\[ g = \sum_{\substack{w \in L(X) \\ |w| =n}} a_w \chi_{w} \]
where $a_w = f(\theta)$ for some $\theta \in \mathcal{F}_w$. Clearly $g \in \mathscr{A}$ and $\|f - g\|_\infty < \varepsilon$.
\end{proof}

\begin{proof}[Proof of Theorem~\ref{t.mfqpsoapplication}]
We consider the full measure set of $\alpha$'s in $\TT^2$ that generate a minimal translation $R_\alpha : \TT^2 \to \TT^2$ and belong to the full measure set determined earlier; compare Proposition~\ref{PROP:T^2-codings}.

By these propositions, the minimal translation $R_\alpha$ admits a natural coding that satisfies the Boshernitzan criterion.  As $R_\alpha^k$ is minimal and for any $f \in \mathscr{A}$ has that its level sets have non-empty interior, the $V_x$ are all aperiodic. Thus, by Theorem~\ref{t.BimpliesZM}, every non-constant locally constant sampling function on this subshift generates a potential so that the associated Schr\"odinger spectrum is a Cantor set of zero Lebesgue measure.

Since the coding is natural, each such locally constant function on the subshift corresponds to an elementary function on the torus and the set of functions obtained via this correspondence is ample by Proposition~\ref{prop:locconstAmple}.
This concludes the proof of the theorem.
\end{proof}

\section{A Discussion of Possible Extensions to Higher Dimensions} \label{sec:higher}

\subsection{A Road Map to Treating Larger Values of $d$}

Proposition \ref{PROP:nat coding} is a significant new result that enabled this project and it is natural to wonder how general it is. The plan for such a result is fairly general.
\begin{enumerate}
\item One finds a continued fraction algorithm and obtains S-adic systems from the process applied to a.e.\ vector in the parameter space.
\item One shows that the resulting shift dynamical systems (a.s.) have purely discrete spectrum, and in fact they are measurably isomorphic to a toral rotation and moreover are natural codings thereof.
\end{enumerate}
Step (2) requires
\begin{itemize}
\item An absolutely continuous ergodic invariant measure.
\item The negativity of the second Lyapanov exponent (of the cocycle that gives the S-adic system) with respect to the absolutely continuous invariant measure.
\item A mild additional assumption. For example either of the following two suffices.
\begin{itemize}
\item As in \cite[Theorem B]{FN20} it 
has a \emph{seed point} (\cite[Definition 64]{FN20}) and the second Lyapanov exponent is simple (this is part of the \emph{Pisot condition} \cite[Definition 60]{FN20} in this paper).
\item As in \cite[Theorem 3.1]{BST20}) it has a \emph{periodic Pisot point} (\cite[Definition 2.4]{BST20}) with \emph{positive range} (\cite[Definition 2.5]{BST20}) so that the corresponding S-adic system (which in this case is a substitution dynamical system) has discrete spectrum.
\end{itemize}
\end{itemize}

There are standard approaches to the ergodicity of these algorithms. For example, one can relate the continued fraction algorithm to a flow that is known to be ergodic (see, e.g., \cite{AN93}) or one can show that it or an acceleration satisfies some well known conditions (see, e.g., \cite[Theorem 8]{Schw00}).

The negativity of the second Lyapanov exponent in dimension greater than two is shown via computer assisted proof in Hardcastle \cite{Har02}; see also Berth\'e--Steiner--Thuswaldner \cite{BST2020MathOfComp}.

There is a general strategy \cite{HK00}, but the rigor of these implementations even in dimension 3 is not always complete \cite{Har02}. For the Cassaigne-Selmer algorithm, one can appeal to the 2 dimensional Selmer algorithm (which it is conjugate to) and quote \cite{Schw04} (which appeals to \cite{Schw01} where the result is proven for the closely related Baldwin algorithm) for a proof without computer assistance.

\subsection{A Brief Discussion of the Case $d = 3$}

For translations on $\TT^3$, the $4$-dimensional Brun algorithm is a natural candidate for the strategy outlined above, and \cite[Section 6.4]{BST20} collects (most of) the necessary inputs.
An analogue of Proposition~\ref{PROP:nat coding} for the Brun algorithm requires one to verify that the second Lyapunov exponent related to the cocycle induced by $A$ on $\Delta$ is negative.  The negativity of the second exponent is unclear to us. In particular \cite{Har02} experimentally studies this question but is not entirely rigorous.\footnote{``Note that I use the term ``proof" here, despite the fact that I do not attempt to control round-off errors. I will leave the issue of whether the term ``proof" is appropriate to the individual reader." \cite[Page 132 bottom of left hand side]{Har02}.} The other assumptions of \cite[Theorem 3.1]{BST20} are verified in the paragraph before \cite[Theorem 6.7]{BST20}.
 The result in \cite[Theorem 6.7]{BST20} states the following.

\begin{prop}
\label{PROP:Brun-coding}
Let $\boldphi$ be the substitutive realization of the Brun algorithm. For $\nu_B$-almost every $\bold{x} \in \Delta$, the subshift $(X({\boldphi(\bold{x})}),S)$ is a natural coding of $(\TT^3, R_{ \pi( \bold{x} ) })$.
\end{prop}

Given the indeterminate status of Proposition~\ref{PROP:Brun-coding}, we regard the following problem as an interesting question for future study.

\begin{conj}\label{CONJ:T^3-codings}
For almost every $\alpha \in \TT^3$, the toral translation $(\TT^3, R_\alpha)$ admits a natural coding that satisfies Boshernitzan's criterion.
\end{conj}

The idea of proof of Conjecture~\ref{CONJ:T^3-codings} relies on Proposition~\ref{PROP:Brun-coding} and thereby on the question whether the second Lyapunov exponent associated to the Brun algorithm is indeed negative---compare the discussion in the previous subsection.  Assuming Proposition~\ref{PROP:Brun-coding}, we can prove Conjecture~\ref{CONJ:T^3-codings} following the same lines as for Proposition~\ref{PROP:T^2-codings}. Here we make use of the observation that the substitution $\tau = \beta_{12} \circ \beta_{23} \circ \beta_{34} \circ \beta_{41}$ is primitive, which can be seen from a direct calculation (indeed $i \triangleleft \tau^2(j)$ for every $i$ and $j$); compare the discussion in \cite{BST20} preceding Theorem~6.7. With $\bold{x}$ the right Perron Frobenius eigenvector of $M_{\tau}$, we have $\phi(\bold{x}) = \tau^{\infty}$. A word builder can be constructed as follows. If $\tau_{[n+1, n+3]} = \beta_{14} \circ \beta_{13} \circ \beta_{12}$, then $a$ can precede $b$ at level $n$ only if $ab \in \{11,12,13,14,21,31,41\}$. From this we can verify that $\beta_{14} \circ \beta_{13} \circ \beta_{12} \circ \tau^2$ is a word builder, irrespective of its position in a sequence $(\tau_n)_{n \in \N_0}$.

\end{document}